\documentclass[12pt]{amsart}
\usepackage{amsfonts}

\newtheorem{theorem}{Theorem}
\theoremstyle{plain}

\newtheorem{corollary}{Corollary}

\newtheorem{definition}{Definition}

\newtheorem{proposition}{Proposition}
\newtheorem{remark}{Remark}

\numberwithin{equation}{section}
\input{tcilatex}

\begin{document}
\title{Hurwitz complete sets of factorizations in the modular group and the
classification of Lefschetz elliptc fibrations over the disk }
\author{C. Cadavid$^{a}$}
\address{a \ Corresponding Author: Carlos Cadavid, {\small {Universidad
EAFIT, Departamento de Ciencias B\'{a}sicas, Bloque 38, Office 417 }}Carrera
49 No. 7 Sur -50, Medell\'{\i}n, Colombia, Phone: (57)(4)-2619500 Ext 9790,
Fax:(57)(4) 2664284}
\email{ccadavid@eafit.edu.co.}
\author{J. D. V\'{e}lez$^b$}
\address{b \ Juan D. V\'{e}lez, Escuela Matematicas, Universidad Nacional,
Medell\'{\i}n Colombia}
\email{jdvelez14@gmail.com }
\author{Luis F. Moreno$^c$}
\address{c \ Lus Moreno l, EAFIT, Departamento de Ciencias B\'{a}sicas,
Bloque 38, Office 417 Medell\'{\i}n Colombia}
\email{lmorenos@eafit.edu.co.}
\thanks{MSC 58E05, 35K05}

\begin{abstract}
Given any matrix $B$ in $SL(2,\mathbb{Z})$, we will describe an algorithm
that provides at least one elliptic fibration over the disk, relatively
minimal and Lefschetz, within each topological equivalence class, whose
total monodromy is the conjugacy class of $B$.
\end{abstract}

\maketitle

\section{Introduction}

\label{sec-introduccion}

Locally holomorphic fibrations have received a great deal of attention due
to the close relationship that exists between a 4-dimensional manifold $M$
admitting a symplectic form and the existence of locally holomorphic
fibrations over $M$ (see \cite{Auroux}, \cite{Gompf-2005}). Such fibrations
have been studied extensively by several authors: Over the sphere, by
Moishezon \cite{Moishezon}, and over closed surfaces of arbitrary genus by
Matsumoto \cite{Matsumoto}. Their classification over the disk, for the case
when the total space is two dimensional, is carried out in \cite{Natanzon}, 
\cite{Khovanskii-Zdravkovska}.

In \cite{Cadavid-Velez}, the authors studied distinguished factorizations in 
$SL(2,\mathbb{Z})$ in terms of conjugates of the matrix $U=\left[ 
\begin{array}{cc}
1 & 1 \\ 
0 & 1%
\end{array}%
\right] $, which naturally arise as the monodromy around a singular fiber in
an elliptic fibration. In that article, it is proved that if $M$ is one of
the matrices in the \textit{Kodaira's list}, and if $M=G_{1}\cdots G_{r}$
where each $G_{i}$ is a conjugate of $U$ in $SL(2,\mathbb{Z})$, then after
applying a finite sequence of Hurwitz moves, the product $G_{1}\cdots G_{r}$
can be transformed into another product of the form $H_{1}\cdots
H_{n}G_{n+1}^{\prime }\cdots G_{r}^{\prime }$ where $H_{1}\cdots H_{n}$ is
some fixed shortest factorization of $M$ in terms of conjugates of $U$, and $%
G_{n+1}^{\prime }\cdots G_{r}^{\prime }=Id_{2\times 2}$. We used this result
to obtain necessary and sufficient conditions under which a relatively
minimal elliptic fibration over the disk $D$ without multiple fibers, $\phi
:S\rightarrow D$, admits a weak deformation into another such fibration
having only one singular fiber.

In general, the problem of classification of elliptic fibrations over $D$
which are \textit{relatively minimal and Lefschetz strict} (see definition %
\ref{estricta}), up to topological equivalence, is equivalent to the problem
of studying the set 
\begin{equation*}
\left\{ (g_{1},\ldots ,g_{n}):n\geq 0\text{ and }g_{i}\in SL(2,\mathbb{Z})%
\text{ is a conjugate of }U\right\} \,,
\end{equation*}%
where two $n$-tuples are identified if one can be obtained from the other by
a finite sequence of Hurwitz moves followed by conjugation (see \cite{Auroux}%
, and Definition \ref{def-cambio-Hurwitz}). A satisfactory answer to the
problem would comprise:

\begin{enumerate}
\item A method by which given any $B\in SL(2,\mathbb{Z})$, one could obtain
a subcollection of the set of all equivalent classes of factorization of $B$
in terms of conjugates of $U$, modulo Hurwitz moves, has at least one
representative in this subcollection.

\item An algorithm to decide if two factorizations of $B$ in conjugates of $%
U $ are Hurwitz equivalent.
\end{enumerate}

In this article we construct an algorithm that completely solves the first
of these goals. Similar results were obtained in \cite{Matsumoto} and \cite%
{Moishezon}, for the case where the base is a closed surface.

The second goal seems to be a very difficult problem. It is known that some
cases turned out to be undecidable (see \cite{Teicher-Liberman}).

The article is organized as follows: in Section 2 we introduce the basic
notions concerning elliptic fibrations over the unit disk and their
classification. The central result is theorem \ref{Gompf} which relates the
problem of classifying all of \textit{special} elliptic fibrations over the
disk to the problem of classifying their monodromy representations, up to
conjugation and Hurwitz equivalence, in the modular group. Section 3 is
devoted to the study of the relationship between \textit{special}
factorizations in $PSL(2,\mathbb{Z})$ and their liftings to $SL(2,\mathbb{Z}%
) $. The next section deals with a combinatorial study of Hurwitz
equivalence of special factorization in the modular group. The last section
presents an algorithm for generating a relatively simple $H$-complete set of
special factorizations of any given element in the modular group.

\section{Elliptic fibrations over the disk and Hurwitz equivalence}

\label{sec-fibraciones-elipticas}

\begin{definition}
\label{def-fibracion-eliptica}Let $\Sigma $ be a compact, connected and
oriented smooth two dimensional manifold (with or without boundary). A
topological elliptic fibration over $\Sigma $ is a smooth function $%
f:M\rightarrow \Sigma \,$such that

\begin{enumerate}
\item $M$ is a compact, connected and oriented four dimensional smooth
manifold (with or without boundary).

\item $f$ is surjective.

\item $f(\mathrm{int}(M))=\mathrm{int}(\Sigma )$ and $f(\partial M)=\partial
(\Sigma ).$

\item $f$ has a finite number (possibly zero) of critical values $%
q_{1},\ldots ,q_{n}$ all contained in $\mathrm{int}(\Sigma ).$

\item $f$ is locally holomorphic, that is, for each $p\in \mathrm{int}(M)$
there exists orientation preserving charts from neighborhoods of $p$ and $%
f(p)$, to open sets of $\mathbb{C}^{2}$ and $\mathbb{C}$ (endowed with their
standard orientations), respectively, relative to which $f$ is holomorphic.

\item The preimage of each regular value is a smooth two dimensional
manifold that is closed and connected, and of genus one.
\end{enumerate}
\end{definition}

Two topological elliptic fibrations are regarded equivalent according to the
following definition.

\begin{definition}
Two topological elliptic fibrations $f_{1}:M_{1}\rightarrow \Sigma _{1}$ and 
$f_{2}:M_{2}\rightarrow \Sigma _{2}$ are \emph{topologically equivalent},
written as $f_{1}\sim _{\text{Top}}\ f_{2}$, if there exist orientation
preserving diffeomorphisms $H:M_{1}\rightarrow M_{2}$ and $h:\Sigma
_{1}\rightarrow \Sigma _{2}$ , such that $h\circ f_{1}=f_{2}\circ H$.
\end{definition}

\begin{definition}
\label{estricta}A topological elliptic fibration $f:M\rightarrow \Sigma $
will be called

\begin{enumerate}
\item \emph{Relatively minimal} if none of its fibers contains an embedded
sphere with selfintersection $-1.$

\item \emph{Lefschetz strict} if for each critical point $p$ (necessarily
contained in $\mathrm{int}(M)$) of $f$ there exist charts as in condition 5
above relative to which $f$ takes the form $(z_{1},z_{2})\rightarrow
z_{1}^{2}+z_{2}^{2}$, and $f$ is injective when restricted to the set of
critical points.
\end{enumerate}
\end{definition}

If $f:M\rightarrow \Sigma $ satisfies both conditions, we will say that $f$
is a \emph{special fibration. }

We notice that being \textit{special} is preserved by topological
equivalence.

\emph{In what follows we will only consider special elliptic fibrations over
the closed unit disk, }$D=\{z\in C:|z|\leq 1\}$\emph{, endowed with its
standard orientation. }

\begin{definition}
\label{conjugado}Let $G$ be a group. Any $n$-tuple of elements of $G,$ $%
\alpha =(g_{1},\ldots ,g_{n})$, $n\geq 0$, will be called a \emph{%
factorization}. The only $0$-tuple (the empty tuple) will be denoted by $($ $%
)$. The element $g_{1}\cdots g_{n}$ will be called the \emph{product} of the
factorization, and will be denoted by \textrm{prod}$(\alpha )$. When $\alpha 
$ is empty, we define its product as the identity element of $G$.

Given any $g$ in $G$, we will say that $\alpha $ is \emph{a factorization }%
of $g$ if its product is equal to $g.$

If $A\subset G$, $F(A,G)$ will denote the set formed by all factorizations
in $G$ whose entries are all in $A.$
\end{definition}

We will be interested in the case where $G$ is $SL(2,\mathbb{Z})$ and $%
A=C(U) $ is the set of all conjugates of the element $U=%
\begin{pmatrix}
1 & 1 \\ 
0 & 1%
\end{pmatrix}%
.$ ($U$ represents the monodromy around a critical point in any special
fibration, as explained below.)

\begin{definition}
We will say that a factorization in $SL(2,\mathbb{Z})$ is \emph{special} if
it belongs to $F(C(U),SL(2,\mathbb{Z}))$. This set will be denoted simply by 
$F(U).$
\end{definition}

\begin{definition}
\label{def-cambio-Hurwitz} Let $G$ be a group, and $n\geq 2$. For any
integer $1\leq i\leq n-1$, a Hurwitz right move, at position $i,$ is the
function $H_{i}:G^{n}\rightarrow G^{n}$ defined as 
\begin{equation*}
H_{i}(g_{1},\ldots ,g_{i},g_{i+1},\ldots ,g_{n})=(g_{1},\ldots
,g_{i-1},g_{i+1},g_{i+1}^{-1}g_{i}g_{i+1},g_{i+2},\ldots ,g_{n})\,.
\end{equation*}%
The inverse function is called a Hurwitz left move, at position $i$, which
is given by 
\begin{equation*}
H_{i}^{-1}(g_{1},\ldots ,g_{i},g_{i+1},\ldots ,g_{n})=(g_{1},\ldots
,g_{i-1},g_{i}g_{i+1}g_{i}^{-1},g_{i},g_{i+2},\ldots ,g_{n})\,.
\end{equation*}%
When $H_{i}(g_{1},\ldots ,g_{n})=(g_{1}^{\prime },\ldots ,g_{n}^{\prime })$
(resp. $H_{i}^{-1}(g_{1},\ldots ,g_{n})=(g_{1}^{\prime },\ldots
,g_{n}^{\prime })$ ) we will say that $(g_{1}^{\prime },\ldots
,g_{n}^{\prime })$ is obtained from $(g_{1},\ldots ,g_{n})$ by a Hurwitz
right move (respectively, by a Hurwitz left move) at position $i.$

\label{clave}If $\alpha ^{\prime }=(g_{1}^{\prime },\ldots ,g_{m}^{\prime })$
is obtained from $\alpha =(g_{1},\ldots ,g_{n})$ by a successive
applications of finite Hurwitz moves, we will say that $\alpha $ and $\alpha
^{\prime }$ are $H$-\emph{equivalent}, which we denote by $\alpha \sim
_{H}\alpha ^{\prime }.$ In this case, it follows immediately that $n=m$ and $%
g_{1}^{\prime }\cdots g_{n}^{\prime }=g_{1}\cdots g_{n}$, and therefore
their product is the same. If, moreover, there exists an element $h$ such
that $\alpha \sim _{H}(h^{-1}g_{1}h,\ldots ,h^{-1}g_{n}h),$ we will say that 
$\alpha $ and $\alpha ^{\prime }$ are $C+H$-\emph{equivalent.} This will be
denoted by $\alpha \sim _{C+H}\alpha ^{\prime }.$
\end{definition}

The set of classes $F(U)/\sim _{H},$ and $F(U)/\sim _{C+H}$ will be denoted
by $\varepsilon _{H}$, and $\varepsilon _{C+H}$, respectively. It is clear
that being $C+H$-equivalent is weaker than being $H$-equivalent.

\subsection{Hurwitz complete sets}

As in definition \ref{conjugado}, $C(B)$ denotes the conjugacy class in $%
SL(2,\mathbb{Z})$ of the matrix $B$. Let us notice that if $\alpha
=(G_{1},\ldots ,G_{r})$ in $F(U)$ has product $B$ , then any other element $%
\alpha ^{\prime }$ in the $H$-equivalence class of $\alpha $ also has
product $B$. On the other hand, if $\alpha ^{\prime }$ is just $C+H$%
-equivalent to $\alpha $, then its product belongs to the conjugacy class of 
$B$.

\begin{definition}
\label{H completo} For any matrix $B$ in $SL(2,\mathbb{Z})$, a subset of $%
F(U)$ will be called $H$-complete (respectively, $H+C$-complete) if it
contains at least one representative within each class of equivalence under
the relation $\sim _{H},$ (respectively, under $\sim _{C+H}$).
\end{definition}

Let $f:M\rightarrow D$ be any special fibration over the disk. Let us denote
by $q_{0}$ the point $(1,0)$, and by $C$ the boundary of the disk with its
standard counterclockwise orientation. As usual, 
\begin{equation*}
\rho :\pi _{1}(D-\{q_{1},\ldots ,q_{n}\},q_{0})\rightarrow SL(2,\mathbb{Z})
\end{equation*}%
will stand for the \emph{monodromy representation} where we have identified
the mapping class group of $T^{2},$ a fixed model of the regular fiber, with 
$SL(2,\mathbb{Z}).$ The mapping $\rho $ is an anti-homomorphism determined
by its action on any basis of the rank $n$ free group $\pi
_{1}(D-\{q_{1},\ldots ,q_{n}\},q_{0})$. We may take $\{[\gamma _{1}],\ldots
,[\gamma _{n}]\}$ the standard basis consisting of the classes of clockwise
oriented, pairwise disjoint arcs where each $\gamma _{i}$ surrounds
exclusively the critical value $q_{i}$, $i=1,\ldots ,n.$ We may choose the $%
\gamma _{i}$'s in such a way that (for an appropriate numbering of the $%
q_{i} $'s) the product $[\gamma _{1}]\cdots \lbrack \gamma _{n}]$ equals the
class of $C$. The conjugacy class in $SL(2,\mathbb{Z})$ of $\rho ([C])$ is
called the \emph{total monodromy} of the fibration. It can be readily seen
that this is a well defined notion.

\begin{remark}
If $f:M\rightarrow D$ is any special fibration over the disk, since each
singular fiber has a single ordinary double point (of type $I_{1}$, in
Kodaira%
\'{}%
s classification \cite{Kodaira} ) the monodromy around any of these fibers
is in the conjugacy class of $U=\left( 
\begin{array}{cc}
1 & 1 \\ 
0 & 1%
\end{array}%
\right) $ in $SL(2,\mathbb{Z}).$
\end{remark}

Special fibrations over $D$ can be classified up to conjugation and Hurwitz
moves. More precisely:

\begin{theorem}
\label{Gompf}Let $f_{1}:M_{1}\rightarrow D$ and $f_{2}:M_{2}\rightarrow D$
be two special fibrations. Let us fix monodromy representations $\rho $ and $%
\rho ^{\prime }$, and basis $\{[\gamma _{1}],\ldots ,[\gamma _{n}]\}$, and $%
\{[\gamma _{1}^{\prime }],\ldots ,[\gamma _{n}^{\prime }]\},$ for $f_{1}$
and $f_{2}$, respectively. Let $g_{i}=\rho ([\gamma _{i}])$, and $%
g_{i}^{\prime }=\rho ^{\prime }([\gamma _{i}^{\prime }]).$ Then, $f_{1}$ and 
$f_{2}$ are topologically equivalent if and only if $\alpha =(g_{1},\ldots
,g_{n})$ and $\alpha ^{\prime }=(g_{1}^{\prime },\ldots ,g_{n}^{\prime })$
are equivalent under the equivalence relation $\sim _{C+H}$(Definition \ref%
{clave}).
\end{theorem}

For a proof see \cite{Gompf-2006}.

Hence, the elements of $\varepsilon _{C+H}$ are in bijective correspondence
with topological equivalency classes of special fibrations over the disk.
Therefore, in order to classify these fibrations, it suffices to describe
the elements of $\varepsilon _{C+H}.$ In this article \emph{we present an
algorithm that for any given matrix }$B$\emph{\ in }$SL(2,Z)$\emph{\
produces an }$H$-\emph{complete set of factorizations of }$B$\emph{. }In
general, this set could be redundant in the sense that it might contain more
that one representative in some equivalence classes. Since $C+H$-equivalence
is weaker than $H$-equivalence, it is clear that this set is also $H+C$%
-complete. \emph{Therefore, for any given }$B,$\emph{\ this algorithm will
provide at least one special elliptic fibration over the disk within each
topological equivalence class, whose total monodromy is the conjugacy class
of }$B$\emph{.}

\section{Special factorizations in $PSL(2,\mathbb{Z})$}

Even though it is well known that $SL(2,\mathbb{Z})$ is generated by the
matrices 
\begin{equation*}
S=%
\begin{pmatrix}
0 & -1 \\ 
1 & 0%
\end{pmatrix}%
\hspace{10pt}\text{and}\hspace{10pt}U=%
\begin{pmatrix}
1 & 1 \\ 
0 & 1%
\end{pmatrix}%
,
\end{equation*}%
it is important for our purposes that a decomposition of any matrix in $SL(2,%
\mathbb{Z})$ as product of powers of $S$ and $U$ (or equivalently, as a
product of powers of $S$ and $R=SU$) can be achieved algorithmically. This
is the content of the next proposition.

\begin{proposition}
Every matrix in $SL(2,\mathbb{Z})$ can be written as a product of powers of $%
S$ and $U$. Moreover, there is an algorithm that given any matrix $B$ in $%
SL(2,\mathbb{Z})$ yields one of such factorizations.
\end{proposition}

\begin{proof}
For any matrix $A$, $U^{n}A$ is the matrix obtained from $A$ by performing
the row operation corresponding to adding $n$ times the second row to the
first, while $SA$ is the matrix obtained from $A$ by performing the row
operation corresponding to interchanging the first and second row, and
multiplying the first row by $-1.$

For any matrix $B=%
\begin{pmatrix}
a & b \\ 
c & d%
\end{pmatrix}%
$, since $\det (B)=1,$ the entries $a$ and $c$ must be relatively prime. If $%
\left\vert c\right\vert <\left\vert a\right\vert ,$ by the euclidean
algorithm, if $a=cn+r$, then by premultiplying by $U^{-n}$ we obtain a
matrix of the form $U^{-n}B=%
\begin{pmatrix}
r & b^{\prime } \\ 
c & d%
\end{pmatrix}%
$

with $b^{\prime }=$ $b-nd.$ In case where $\left\vert a\right\vert
<\left\vert c\right\vert $, we may first multiply by $S$ to interchange the
rows. Thus, in any case, premultiplying by $U^{-n},$ or by $U^{-n}S$, has
the effect of putting $B$ in the form $%
\begin{pmatrix}
r & b^{\prime } \\ 
c & d%
\end{pmatrix}%
$, where $\mathrm{lcd}(a,c)=\mathrm{lcd}(c,r)$ ($\mathrm{lcd}$ denotes the
lest common divisor). By successively premultiplying by $S,$ and suitable
powers of $U,$ we may transform $B$ into a matrix of the form $B^{\prime }=%
\begin{pmatrix}
\pm 1 & m \\ 
0 & k%
\end{pmatrix}%
.$ That is, $B^{\prime }=PB$, where $P$ is a product of $S$ and powers of $U$%
. Since $B^{\prime }$is in $SL(2,\mathbb{Z})$, $k$ must be equal to $\pm 1.$
Therefore, $B^{\prime }=\pm I_{2}U^{\pm m}.$ Since $S^{2}=-I_{2}$, then $%
B=P^{-1}(\pm I_{2})U^{\pm m}.$
\end{proof}

The modular group, $SL(2,\mathbb{Z})/\{\pm I_{2}\},$ will be denoted by $%
PSL(2,\mathbb{Z)}$. For the sake of brevity, we will denote this group
simply by $\mathcal{M}.$The classes of $S,U$ and $R$ will be denoted by $%
\omega ,u$ and $b$, respectively. Note that $b=\omega u.$ It is a well known
fact 
\begin{equation*}
\mathcal{M}\mathbb{=}\left\langle \omega ,b\left\vert \ \omega
^{2}=b^{3}=1\right. \right\rangle .
\end{equation*}

The following corollary is an immediate consequence of the previous
proposition.

\begin{corollary}
There is an algorithm that expresses any element in $\mathcal{M}$ as a
product of positive powers of $\omega $ and $b.$
\end{corollary}

Let $\pi :SL(2,\mathbb{Z})\rightarrow \mathcal{M}$ denote the canonical
homomorphism to the quotient.

\begin{definition}
A factorization $\alpha =(g_{1},\ldots ,g_{n})$ in $\mathcal{M}$ will be
called \emph{special} if each $g_{i}$ is a conjugate of $u.$

A special factorization $\alpha =(A_{1},\ldots ,A_{n})$ in $SL(2,\mathbb{Z})$
will be called a \emph{lift} \emph{of }$\alpha $, if $\pi (A_{i})=g_{i}$ for
each $i.$
\end{definition}

We observe that each special factorization $\alpha =(g_{1},\ldots ,g_{n})$
in $\mathcal{M}$ has exactly one lift. Indeed, if $g_{i}=a_{i}ua_{i}^{-1}$,
then its preimages are $\pm A_{i}UA_{i}^{-1},$ where $A_{i}$ is any preimage
of $a_{i}.$ But only $A_{i}UA_{i}^{-1}$ is a conjugate of $U$, since the
trace$(-A_{i}UA_{i}^{-1})=-2$, and every conjugate of $U$ has trace $2.$ The
lift $\alpha $ will be denoted by $\mathrm{lift}(\alpha )$.

Now, in $\mathcal{M},$ if $\alpha ^{\prime }$ is obtained from $\alpha $ by
performing a Hurwitz move, then $\mathrm{lift}(\alpha ^{\prime })$ can be
obtained from $\mathrm{lift}(\alpha )$ by the corresponding move.
Reciprocally, Hurwitz moves in $SL(2,\mathbb{Z})$ can be transformed into
Hurwitz moves in $\mathcal{M}$ via $\pi $. Therefore, $\alpha \sim
_{H}\alpha ^{\prime }$ if and only if $\mathrm{lift}(\alpha )\sim _{H}%
\mathrm{lift}(\alpha ^{\prime }).$ From this, it follows that $H$-complete
sets for a matrix in $SL(2,\mathbb{Z})$ can be obtained from $H$-complete
sets for $\pi (B)$ in $\mathcal{M}.$ More precisely:

\begin{proposition}
Let $A$ be an element of $SL(2,\mathbb{Z})$. If $\mathcal{S}$ is an $H$%
-complete set for $\pi (A)$ then the collection 
\begin{equation*}
\mathcal{R}=\{\text{\textrm{lift}}(\alpha ):\alpha \in \mathcal{S}\text{ and 
\textrm{prod(lift}}(\alpha ))=A\}
\end{equation*}%
is an $H$-complete set for $A.$
\end{proposition}

\begin{proof}
The proposition follows from the obvious observation that if $\alpha \sim
_{H}\alpha ^{\prime }$ in $SL(2,\mathbb{Z})$ then \textrm{prod(}$\alpha )=$%
\textrm{prod(}$\alpha ^{\prime })$.
\end{proof}

\section{$H$-complete sets in $\mathcal{M}$}

In this section, $\mathcal{M}$ will be identified with the free product 
\begin{equation*}
\mathbb{Z}_{2}\ast \mathbb{Z}_{3}=\left\langle \omega ,b\left\vert \ \omega
^{2}=b^{3}=1\right. \right\rangle .
\end{equation*}%
There is a unique automorphism $\phi $ of $\mathcal{M}$ that sends $\omega $
into itself and $b$ into $b^{2}.$Let us denote by $c_{b}:\mathcal{M}%
\rightarrow \mathcal{M}$ conjugation by $b$, i.e., $c_{b}(z)=bzb^{-1},$ and
by $h$ the composition $h=c_{b}\circ \phi .$ The problem of finding $H$%
-complete sets in $\mathcal{M}$ in terms of conjugates of $u=\omega b$ is
equivalent, via $h$, to the problem of finding $H$-complete sets of elements
in terms of conjugates of $h(u)=b\omega b.$

It is important to have a symbol for the empty word: We will denote it by $%
1. $

It is a standard fact that each element $a$ in $\mathcal{M}$ can be written
uniquely as a product $a=t_{k}\cdots t_{1}$, where each $t_{i}$ is either $%
\omega ,b,$ or $b^{2}$ and no consecutive pair $t_{i}t_{i+1}$ is formed
either by two powers of $b$ or two copies of $\omega $. .We call the product 
$t_{k}\cdots t_{1}$ the \emph{reduced expression}\textit{\ of }$a,$ and we
call $k$ the \emph{length} of $a,$ denoted by $l(a).$ Let $z=t_{1}^{\prime
}\cdots t_{l}^{\prime }$ be the reduced expression of $z.$ If exactly the
first $m-1$ terms of $z$ cancel with those of $a,$ i.e. $t_{i}^{\prime
}=t_{i}^{-1}$, for $1\leq i\leq m-1,$ and if $m\leq \min (k,l),$ then $%
az=t_{k}\cdots t_{m}t_{m}^{\prime }\cdots t_{l}^{\prime }$ and $%
t_{m}t_{m}^{\prime }$ has to be equal to a non trivial power of $b.$ This is
because if $t_{m}$ were not a power of $b$ then it would have to be $\omega $
and therefore $t_{m-1}$ would be a first or second power of $b,$ and so
would be $t_{m-1}^{\prime }.$ Hence, $t_{m}^{\prime }$ would also have to be 
$\omega $ but in this case there would be $m$ instead of $m-1$ cancellations
at the juncture of $a$ and $z$. Thus, $t_{m}$ and $t_{m}^{\prime }$ are both
powers of $b$ and since there are exactly $m-1$ cancellations their product
must be non trivial. Thus, the reduced expression for $az$ is of the form%
\begin{equation}
az=t_{k}\cdots t_{m+1}b^{r}t_{m+1}^{\prime }\cdots t_{l}^{\prime }\text{, }%
r=1\text{ or }2\text{, \ if }m\leq \min (k,l).  \label{F0}
\end{equation}

Let $s_{1}$ denote the element $b\omega b$. The shortest conjugates of $%
s_{1} $ in $\mathcal{M}$ are precisely $s_{0}=b^{2}(b\omega b)b=\omega b^{2}$
and $s_{2}=b(b\omega b)b^{2}=b^{2}\omega $. The element $s_{1}$ is trivially
a conjugate of itself of length $3$. It can be easily seen that if $g$ is a
conjugate of greater length, its reduced expression is of the form $%
Q^{-1}s_{1}Q,$ where $Q$ is a reduced word that begins with $\omega $ (see 
\cite{Friedman-Morgan}), and $l(g)=2l(Q)+3.$ We will call a conjugate of $%
s_{1}$ (\textit{conjugate} will always mean conjugate of $s_{1}$ in $%
\mathcal{M}$) \emph{short} if $g\in \{s_{0},s_{1},s_{2}\},$ otherwise it
will be called\emph{\ long}.

The following notion is the key ingredient for understanding the reduced
expression of a product of conjugates of $s_{1}.$

\begin{definition}
We will say that two conjugates $g$ and $h$ of $s_{1}$ \emph{join well }if $%
l(gh)\geq \max (l(g),l(h)).$ Otherwise, we say they \emph{join badly}.
\end{definition}

The notion of being a \textit{special} factorization will be used in the
following sense:

\begin{definition}
A factorization $\alpha =(g_{1},\ldots ,g_{n})$ in $\mathcal{M}$ is called 
\emph{special} if each $g_{i}$ is a conjugate of $s_{1}$. We say $\alpha $
is \emph{well jointed} if each pair of elements $g_{i},g_{i+1}$ join well.
Otherwise, we say that $\alpha $ is \emph{badly jointed.}

The empty factorization will be regarded as being special, and well jointed.
Special factorizations with just one element will also be regarded as well
jointed.
\end{definition}

\begin{remark}
\label{notilla} We notice that the following identities hold:

$s_{2}s_{2}=b^{2}\omega b^{2}\omega $, $s_{1}s_{1}=b\omega b^{2}\omega b$,

$s_{0}s_{0}=\omega b^{2}\omega b^{2}$, $s_{2}s_{1}=b^{2}\omega b\omega b$,

and, $s_{1}s_{0}=b\omega b\omega b^{2}$ and $s_{0}s_{2}=\omega b\omega $.

Hence, the corresponding factorizations in each case are well jointed. On
the other hand, since $s_{0}s_{1}=s_{1}s_{2}=s_{2}s_{0}=b$, the
corresponding factorizations are badly joined. \label{juntan-mal}
\end{remark}

The following propositions will be useful for the proof of one of the main
results used for the construction of $H$-complete sets.

\begin{proposition}
\label{prop-disminuir-longitud-pares} Let $g_{1},g_{2}$ be conjugates of $%
s_{1}$ such that $g_{1},g_{2}$ do not joint well. Then:

\begin{enumerate}
\item $g_{1},g_{2}$ are short conjugates or

\item $(g_{1},g_{2})$ may be transformed by a Hurwitz move into a new pair $%
(h_{1},h_{2})$ such that $\max \{0,l(h_{1})-3\}+\max \{0,l(h_{2})-3\}<\max
\{0,l(g_{1})-3\}+\max \{0,l(g_{2})-3\}$.
\end{enumerate}
\end{proposition}

\begin{proof}
It follows from the proof of Proposition 4.15, \cite{Friedman-Morgan}.
\end{proof}

\begin{proposition}
\label{prop-transformacion} Every spacial factorization $\alpha
=(g_{1},\ldots ,g_{n})$ can be trasformed by Hurwitz moves into a
factorization $\beta =(h_{1},\ldots ,h_{n})$(necessarily special, and with
the same number of factors), satisfying:

\begin{itemize}
\item[i)] Each $h_{i}$ is short, or

\item[ii)] $\beta $ is well jointed and at least one of the $h_{i}^{\prime
}s $ is long.
\end{itemize}
\end{proposition}

\begin{proof}
(See \cite{Friedman-Morgan})
\end{proof}

\begin{proposition}
\label{prelema-prop-corto-Hurwitz-equivalente} Every factorization $%
(g_{1},g_{2},g_{3})$ in which each $g_{i}$ is short, and where $g_{1},g_{2}$
join badly, is $H$-equivalent to a factorization $(g_{1}^{\prime
},g_{2}^{\prime },g_{3}^{\prime })$, where each $g_{i}^{\prime }$ is short,
and $g_{2}^{\prime },g_{3}^{\prime }$ join badly.
\end{proposition}

\begin{proof}
The only pairs of short conjugates that do not join well are $%
(s_{0},s_{1}),(s_{1},s_{2}),$

$(s_{2},s_{0})$. It follows that for each $s_{i}$ there exists an $s_{j}$
such that $(s_{j},s_{i})$ does not join well. We also notice that any two of
these pairs are $H$-equivalent. Hence, for $g_{3},$ there is $s_{j}$ such
that $(s_{j},g_{3})$ does not join well. Therefore, after a Hurwitz move
performed on the pair $(g_{1},g_{2})$, transforming it into $(g_{1}^{\prime
},g_{2}^{\prime }),$ with $g_{2}^{\prime }=s_{j}$, then, the factorization $%
(g_{1}^{\prime },g_{2}^{\prime },g_{3}^{\prime })$ with $g_{3}^{\prime
}=g_{3}$, is Hurwitz equivalent to $(g_{1},g_{2},g_{3}),$ and $%
(g_{2}^{\prime },g_{3}^{\prime })$ does not join well.
\end{proof}

\begin{proposition}
\label{lema-prop-corto-Hurwitz-equivalente}
\end{proposition}

\begin{enumerate}
\item Every factorization $(g_{1},\ldots ,g_{n})$ where each $g_{i}$ is a
short conjugate, and where not all pairs of elements $g_{i},g_{i+1}$ join
well, is $H$-equivalent to a factorization $(g_{1}^{\prime },\ldots
,g_{n}^{\prime }),$ where $(g_{n-1}^{\prime },g_{n})$ join badly.

\item Every factorization $(g_{1},\ldots ,g_{n})$ in short conjugates where
not all pairs of elements $g_{i},g_{i+1}$ join well is $H$-equivalent to a
factorization $(g_{1}^{\prime },\ldots ,g_{n}^{\prime })$ in which $%
(g_{n-1}^{\prime },g_{n})=(s_{0},s_{1})$.
\end{enumerate}

\begin{proof}
For each factorization $\alpha =(g_{1},\ldots ,g_{n})$ $(n\geq 2)$ in short
conjugates where not all pairs of elements $g_{i},g_{i+1}$ join well we
associate the integer $k(\alpha )=n-\max \{r:(g_{r},g_{r+1})\ $does not join
well$\}$. The proof proceeds by induction on $k$. If $k=1$, then $%
(g_{n-1},g_{n})$ join badly, and the result follows. For $k_{0}\geq 1,$ let
us suppose that the result is true for all $\alpha $ such that $k(\alpha
)\leq k_{0}$. Let $\beta =(g_{1},\ldots ,g_{n})$ be a factorization with $%
k(\beta )=k_{0}+1$. This implies that $(g_{n-k_{0}-1},g_{n-k_{0}})$ does not
join well. Applying Proposition \ref{prelema-prop-corto-Hurwitz-equivalente}
we infer that $(g_{n-k_{0}-1},g_{n-k_{0}},g_{n-k_{0}+1})$ is $H$-equivalent
to a factorization $(g_{n-k_{0}-1}^{\prime },g_{n-k_{0}}^{\prime
},g_{n-k_{0}+1}^{\prime })$ in short conjugates, such that $%
(g_{n-k_{0}}^{\prime },g_{n-k_{0}+1}^{\prime })$ join badly. Summarizing,
the original factorization $\beta $ is $H$-equivalent to a factorization in
short conjugates $\beta ^{\prime }=(g_{1}^{\prime },\ldots ,g_{n}^{\prime })$
in which $(g_{n-k_{0}}^{\prime },g_{n-k_{0}+1}^{\prime })$ join badly.
Clearly $k(\beta ^{\prime })<k(\beta )$, thus the proposition holds for $%
\beta ^{\prime }$, i.e., $\beta ^{\prime }$ is $H$-equivalent to another
factorization in short conjugates $\beta ^{\prime \prime }=(g_{1}^{\prime
\prime },\ldots ,g_{n}^{\prime \prime })$ in which $(g_{n-1}^{\prime \prime
},g_{n}^{\prime \prime })$ join badly. We conclude that the result also
holds $\beta ,$ since $\beta $ is Hurwitz equivalent to $\beta ^{\prime
\prime }$. This proves the first statement. The second assertion easily
follows from the fact that all pairs of short conjugates that join badly are 
$H$-equivalent to $(s_{0},s_{1})$.
\end{proof}

\begin{proposition}
\label{prop-corto-Hurwitz-equivalente} Every factorization $(g_{1},\ldots
,g_{n})$ in short conjugates is $H$-equivalent to another factorization in
short conjugates, of the form $(g_{1}^{\prime },\ldots ,g_{m}^{\prime
},s_{0},s_{1},\ldots ,s_{0},s_{1})$, $0\leq m\leq n,$ where there are $%
(n-m)/2$ pairs of $s_{0},s_{1}$, and $(g_{1}^{\prime },\ldots ,g_{m}^{\prime
})$ is well jointed.
\end{proposition}

\begin{proof}
Let $\alpha =(g_{1},\ldots ,g_{n})$ be a factorization in short conjugates.
Each factorization $\beta =(h_{1},\ldots ,h_{n})$ in short conjugates that
is $H$-equivalent to $\alpha $ can be written uniquely as $(h_{1},\ldots
,h_{m},s_{0},s_{1},\ldots ,s_{0},s_{1})$ where there are $r\geq 0$ pairs $%
s_{0},s_{1}$, and where $m\geq 0$ and $(h_{m-1},h_{m})\neq (s_{0},s_{1}),$
if $m\geq 2$. The integer $r$ will be denoted by $r(\beta )$ to indicate its
dependence on $\beta $. Let $\gamma =(g_{1}^{\prime },\ldots ,g_{m}^{\prime
},s_{0},s_{1},\ldots ,s_{0},s_{1})$ be a factorization in short conjugates, $%
H$-equivalent to $\alpha $, such that $r(\gamma )\geq r(\beta )$ for any
other factorization in short conjugates $\beta $, $H$-equivalent to $\alpha $%
. Let us verify that $(g_{1}^{\prime },\ldots ,g_{m}^{\prime })$ is well
jointed. If $(g_{1}^{\prime },\ldots ,g_{m}^{\prime })$ is badly jointed,
and $m\geq 2,$ by the second part of Proposition \ref%
{lema-prop-corto-Hurwitz-equivalente} there would be another factorization
in short conjugates $(g_{1}^{\prime \prime },\ldots ,g_{m}^{\prime \prime })$
$H$-equivalent to $(g_{1}^{\prime },\ldots ,g_{m}^{\prime })$, and such that 
$(g_{m-1}^{\prime \prime },g_{m}^{\prime \prime })=(s_{0},s_{1})$. Hence, $%
\gamma $ would also be (and, therefore $\alpha $), $H$-equivalent to a
factorization in short conjugates $(g_{1}^{\prime \prime },\ldots
,g_{m-2}^{\prime \prime },s_{0},s_{1},\ldots ,s_{0},s_{1})$ with $r(\gamma
)+1\ $pairs $s_{0},s_{1}$, in contradiction with the maximality of $\gamma .$
Thus, $(g_{1}^{\prime },\ldots ,g_{m}^{\prime })$ is well jointed.
\end{proof}

\begin{proposition}
Each special factorization $(g_{1},\ldots ,g_{n})$ is $H$-equivalent to a
factorization of the form $(g_{1}^{\prime },\ldots ,g_{m}^{\prime
},s_{0},s_{1},\ldots ,s_{0},s_{1})$, where there are $r\geq 0$ pairs $%
s_{0},s_{1}$, and where $(g_{1}^{\prime },\ldots ,g_{m}^{\prime })$ is well
jointed. Moreover, $(g_{1}^{\prime },\ldots ,g_{m}^{\prime })$ is a
factorization in short conjugates, whenever $r>0$.
\end{proposition}

\begin{proof}
By Proposition \ref{prop-transformacion}, $(g_{1},\ldots ,g_{n})$ is $H$%
-equivalent to a factorization $\beta =(g_{1}^{\prime },\ldots
,g_{n}^{\prime })$ that either, is well jointed and at least one of the $%
g_{i}^{\prime }s$ is a long conjugate, or it is badly jointed and all $%
g_{i}^{\prime }s$ are short conjugates. In the first case, $\beta $ already
has the desired form, since the fact that the factors join well implies that 
$(g_{n-1}^{\prime },g_{n}^{\prime })\neq (s_{0},s_{1}),$ and consequently $%
r=0$. Now, in case $\beta $ consists of short conjugates that join well,
then it also has already the desired form for the same reason. Hence, let us
suppose that $\beta $ is a factorization in short conjugates that is badly
joined. By Proposition \ref{prop-corto-Hurwitz-equivalente}, this
factorization is $H$-equivalent to another one in short conjugates, of the
form $(g_{1}^{\prime \prime },\ldots ,g_{m}^{\prime \prime
},s_{0},s_{1},\ldots ,s_{0},s_{1}),$ with $(n-m)/2$ pairs $s_{0},s_{1}$, and
where $(g_{1}^{\prime \prime },\ldots ,g_{m}^{\prime \prime })$ is well
jointed.
\end{proof}

An immediate consequence is the following theorem.

\begin{theorem}
\label{central}For each $g\in \mathcal{M}$, the set of all special
factorizations of $g$ having either of the following two forms is $H$%
-complete:

\begin{enumerate}
\item $(g_{1},\ldots ,g_{m},s_{0},s_{1},\ldots ,s_{0},s_{1}),$ where there
are $r>0$ pairs $s_{0},s_{1}$, $(g_{1},\ldots ,g_{m})$ is well jointed, and
each $g_{i}$ is short.

\item $(g_{1},\ldots ,g_{p})$, where this factorization is well jointed.
\end{enumerate}
\end{theorem}

\section{An algorithm to produce $H$-complete sets}

For $h$ in the modular group, let us denote by $WJ(h)$ the set formed by all 
\emph{special factorizations} of $h$ that are well jointed, and by $WJS(h)$
the subset of factorizations in short conjugates. Remember that we regard
the empty factorization $(\ )$ as a \emph{well jointed} \emph{special
factorization of the identity }$1,$ \emph{in short conjugates}.

Since $s_{0}s_{1}=b$ and $b^{3}=1$, we have that $(s_{0}s_{1})^{3k+l}$
equals $1$ if $l=0$, $b$ if $l=1$ and $b^{2}$ if $l=2$. According to Theorem %
\ref{central}, for any fixed element $g$, the union of the following four
sets of factorizations of $g$ is $H$-complete:

\begin{enumerate}
\item $A=\{\alpha :$ $\alpha $ is a \emph{well jointed} \emph{special}
factorization of $g\}$.

\item $B=\{(g_{1},\ldots ,g_{m},s_{0},s_{1},\ldots ,s_{0},s_{1}):$ $%
(g_{1},\ldots ,g_{m})$ is a \emph{well jointed} \emph{special} factorization
of $g$ in\emph{\ short} conjugates and the number of pairs $s_{0},s_{1}$ is
of the form $3k,$ with $k\geq 1\}.$

\item $C=\{(g_{1},\ldots ,g_{m},s_{0},s_{1},\ldots ,s_{0},s_{1}):$ $%
(g_{1},\ldots ,g_{m})$ is a \emph{well jointed} \emph{special} factorization
of $gb^{2}$ in \emph{short} conjugates and the number of pairs $s_{0},s_{1}$
is of the form $3k+1,$ with $k\geq 0\}.$

\item $D=\{(g_{1},\ldots ,g_{m},s_{0},s_{1},\ldots ,s_{0},s_{1}):$ $%
(g_{1},\ldots ,g_{m})$ is a \emph{well jointed} \emph{special} factorization
of $gb$ in \emph{short} conjugates and the number of pairs $s_{0},s_{1}$ is
of the form $3k+2,$ with $k\geq 0\}.$
\end{enumerate}

In consequence:

\begin{remark}
\label{superutil}In order to find an $H$-complete set of special
factorizations of an element $g$ we need i) an algorithm that takes an
element $h$ in the modular group, and produces the set $WJ(h)$, and ii) an
algorithm that extracts the subset $WJS(h)$.
\end{remark}

This second task is trivial, but the first one is less so. The key
ingredient to formulate the algorithm in i) is discussed next.

\begin{definition}
We define the \emph{left part} of short conjugates of $b\omega b$ as $%
left(s_{0})=left(\omega b^{2})=\omega $, $left(s_{1})=left(b\omega
b)=b\omega $, $left(s_{2})=left(b^{2}\omega )=b^{2}\omega $. For long
conjugates, $left(P^{-1}b\omega bP)=P^{-1}b\omega ,$ where $P$ is an element
of the modular group that begins with $\omega $.
\end{definition}

The following result is Lemma 2.4 in \cite{Matsumoto}.

\begin{proposition}
If $(h_{1},\ldots ,h_{n})$ with $n\geq 1$ is a special factorization that is
well jointed, its product $h_{1}\cdots h_{n}$ begins with $left(h_{1})$.
\end{proposition}

\begin{proof}
See \cite{Matsumoto}.
\end{proof}

According to this result, if $(h_{1},\ldots ,h_{n})$ is a well jointed
special factorization of an element $h$ in the modular group, then $h_{1}$
is either one of the following:

\begin{enumerate}
\item $\omega b^{2}$ if $h$ begins with $\omega $,

\item $b\omega b$ if $h$ begins with $b\omega $,

\item $b^{2}\omega $ if $h$ begins with $b^{2}\omega $,

\item $P^{-1}b\omega bP$ if $h$ begins with $P^{-1}b\omega $ for any $P$
that begins with $\omega $.
\end{enumerate}

In particular, the element $1$ has only one well jointed special
factorization, namely the empty factorization. Also, $b$ and $b^{2}$ admit
no well jointed special factorization.

Now we give an algorithm, that we will call \textit{FirstFactor}, which
takes as input any element $h$ in the modular group, with $h$ not in the set 
$\{1,b,b^{2}\}$, and produces all possible candidates to be first factors in
any well jointed special factorization of $h$. This algorithm outputs a set
that contains:

\begin{description}
\item[a] $\omega b^{2},$ if $h$ begins with $\omega $,

\item[b] $b\omega b,$ if $h$ begins with $b\omega $,

\item[c] $b^{2}\omega ,$ if $h$ begins with $b^{2}\omega $

\item[d] For each occurrence of $b\omega $ that is not at the beginning of $%
h $, the element $P^{-1}b\omega bP,$ where $P^{-1}$ is the initial section
of $h$ ending right before the occurrence of $b\omega $ starts.
\end{description}

Then we define another algorithm that we will call \textit{Sibling}. This
algorithm receives as input an ordered pair $((g_{1},\ldots ,g_{n}),z),$
where $(g_{1},\ldots ,g_{n})$ is a special factorization and $z$ is any
element in the modular group. Then, \textit{Sibling} takes the following
actions:

\begin{enumerate}
\item If $z=1,$ then \textit{Sibling} outputs the set $\{((g_{1},\ldots
,g_{n}),z)\}$.

\item If $z$ is $b$ or $b^{2}$, then \textit{Sibling} outputs the empty set $%
\{$ $\}$.

\item If $z$ is different from $1$, $b$ and $b^{2}$, and $(g_{1},\ldots
,g_{n})$ is the empty factorization, then \textit{Sibling} computes the
(necessarily nonempty) set $F=\mathrm{FirstFactor}(z),$ and then outputs the
set $\{((g),g^{-1}z):g\in F\}$.

\item If $z$ is different from $1$, $b$ and $b^{2}$, and $(g_{1},\ldots
,g_{n})$ is not the empty factorization, \textit{Sibling} computes the
(necessarily nonempty) set $F=\mathrm{FirstFactor}(z)$ and then outputs the
set $\{((g_{1},\ldots ,g_{n},g),g^{-1}z):g\in F\ \text{and}\ (g_{n},g)\ 
\text{join well}\}$.
\end{enumerate}

We make the following elementary but important observations:

\begin{enumerate}
\item For each pair $((h_{1},\ldots ,h_{n}),z)$ formed by a factorization
and any element $z$ in the modular group, we call $h_{1}\cdots h_{n}z$ \emph{%
the product of the pair}. Then \textit{Sibling} preserves products, i.e.,
each pair in $\mathit{Sibling}(((g_{1},\ldots ,g_{n}),z))$ has the same
product as the pair $((g_{1},\ldots ,g_{n}),z)$. Notice that this statement
is true even if $z$ is $b$ or $b^{2}$.

\item If $((g_{1},\ldots ,g_{n}),z),$ where $(g_{1},\ldots ,g_{n})$ is a
well jointed special factorization with $n\geq 0$, then the first component
of each element of \textit{Sibling}$(((g_{1},\ldots ,g_{n}),z))$ is a
special factorization that is also well jointed. Notice that this is true
even if $z$ is $b$ or $b^{2}$.

\item 
\begin{enumerate}
\item \textit{Sibling}$(((g_{1},\ldots ,g_{n}),\omega ))=\{((g_{1},\ldots
,g_{n},\omega b^{2}),b)\}$ for any special factorization $(g_{1},\ldots
,g_{n})$ with $n\geq 0$. Therefore 
\begin{equation*}
\mathit{Sibling}(\mathit{Sibling}(((g_{1},\ldots ,g_{n}),\omega )))=\{\ \}
\end{equation*}%
for any special factorization $(g_{1},\ldots ,g_{n})$ with $n\geq 0$.

\item \textit{Sibling}$(((g_{1},\ldots ,g_{n}),z))=\{\ \},$ if $z=b,b^{2}$
and $(g_{1},\ldots ,g_{n})$ is a special factorization with $n\geq 0$.

\item \textrm{Sibling}$(((g_{1},\ldots ,g_{n}),z))=\{((g_{1},\ldots
,g_{n}),z)\},$ if $z$ is $1$ and $(g_{1},\ldots ,g_{n})$ is a special
factorization with $n\geq 0$.

\item Let $z\notin \{1,\omega ,b,b^{2}\}$ and let $g\in \mathrm{FirstFactor}%
(z)$. Let us see that $l(g^{-1}z)<l(z)$. Let $z$ begin with $\omega $ and $%
g=\omega b^{2}$. Then $z$ will be of the form $\omega b^{\delta }Q,$ where $%
\delta =1,2$ and $Q$ is a reduced word that is $1$ or begins with $\omega $.
We have 
\begin{equation*}
g^{-1}z=(\omega b^{2})^{-1}(\omega b^{\delta }\ Q)=(b\omega )(\omega
b^{\delta })Q=b^{\gamma }Q,
\end{equation*}%
where $\gamma $ is $0$ or $2$, and clearly $l(b^{\gamma }Q)<l(\omega
b^{\delta }Q)$. Let $z$ begin with $b\omega $ and $g=b\omega b$. In this
case $z$ is of the form $b\omega Q,$ where $Q$ is a reduced word that is
either $1$ or begins with $b$ or $b^{2}$. We have $g^{-1}z=(b^{2}\omega
b^{2})(b\omega Q)=b^{2}Q$, and clearly $l(b^{2}Q)<l(b\omega Q)$. Let $z$
begin with $b^{2}\omega $ and let $g=b^{2}\omega $. It is clear that $%
l(g^{-1}z)<l(z)$ in this case. Let $z$ begin with $P^{-1}b\omega ,$ where $P$
is a reduced word that begins with $\omega $, and let $g=P^{-1}b\omega bP$.
Then $z$ is of the form $P^{-1}b\omega Q,$ where $Q$ is a reduced word that
is either $1$ or begins with $b^{\delta }$ with $\delta =1,2$. Then $%
g^{-1}z=(P^{-1}b^{2}\omega b^{2}P)(P^{-1}b\omega Q)=P^{-1}b^{2}Q$. Clearly, $%
l(P^{-1}b^{2}Q)<l(P^{-1}b\omega Q)$.
\end{enumerate}
\end{enumerate}

Now we define another routine, that we will call \textit{SiblingSets} that
takes as input a set $S$ whose elements are ordered pairs of the form $%
((g_{1},\ldots ,g_{n}),z)$, and outputs the set $\cup _{s\in S}\mathrm{%
Sibling}(s)$. Notice that \textit{SiblingSets} applied to the empty set
gives the empty set. Finally, we define a routine, that we call \textit{%
WellJointed} that takes an element $h$ in the modular group as input, then
calculates the result of applying $l(h)+1$ times \textit{SiblingSets} to the
set $\{((\ ),h)\}$, i.e. calculates $T=\mathit{SiblingSets}^{l(h)+1}(\{((\
),h)\})$, and then outputs the set formed by the first components of the
ordered pairs in $T$.

By all the observations above, the algorithm $\mathrm{WellJointed}$ finds
all possible well jointed special factorizations of any element $h.$ By
Remark \ref{superutil} this is all we needed in order to find an $H$%
-complete set of special factorizations of an element $g.$

\section{Acknowledgements}

We thank the Universidad Nacional of Colombia and Universidad Eafit for
their invaluable support.

\end{document}